\documentclass[ 11pt]{amsart}

\usepackage[T1]{fontenc}
\usepackage[utf8]{inputenc}
\usepackage{geometry}
\usepackage{amssymb}

\makeatletter
\@namedef{subjclassname@2020}{%
  \textup{2020} Mathematics Subject Classification}
\makeatother

\allowdisplaybreaks[2]
\usepackage{amssymb}

\newtheorem{theorem}{Theorem}[section]
\newtheorem{proposition}[theorem]{Proposition}
\newtheorem{lemma}[theorem]{Lemma}
\newtheorem{corollary}[theorem]{Corollary}
\theoremstyle{definition}

\newtheorem{definition}[theorem]{Definition}

\numberwithin{equation}{section}

\newtheorem{XxmpX}[theorem]{Example} 
\newenvironment{example}    
  {%
   \pushQED{\qed}\begin{XxmpX}}
  {\popQED\end{XxmpX}}

\newtheorem{XxmpY}[theorem]{Remark} 
\newenvironment{remark}    
  {%
   \pushQED{\qed}\begin{XxmpY}}
  {\popQED\end{XxmpY}}

\begin{document}
\title[Splittable Jordan homomorphisms  and commutator ideals]{Splittable Jordan homomorphisms  and commutator ideals}

\author{Matej Brešar} 

\address{Faculty of Mathematics and Physics, University of Ljubljana \&
Faculty of Natural Sciences and Mathematics, University of Maribor \& IMFM, Ljubljana, Slovenia}
\email{matej.bresar@fmf.uni-lj.si}

\thanks{Partially supported by ARIS Grant P1-0288}



\begin{abstract} 
We define a Jordan homomorphism $\varphi$ from a ring $R$ to a ring $R'$ to be splittable if the ideal (of the subring generated by the image of $\varphi$) generated by all $\varphi(xy)-\varphi(x)\varphi(y)$, $x,y\in R$, has trivial intersection with the ideal generated by 
 all $\varphi(xy)-\varphi(y)\varphi(x)$, $x,y\in R$.
Our main result states that a splittable Jordan homomorphism is the sum of a homomorphism and an antihomomorphism on the commutator ideal. As  applications, we obtain  results  that   give  new insight into the question of the structure of Jordan homomorphisms on some classes of rings.
\end{abstract}

\maketitle

\section{Introduction}

Let $R$ and $R'$ be associative rings. 
An additive map
$\varphi:R\to R'$ is called a {\em Jordan homomorphism} if
\begin{equation}
\label{j1}\varphi(x^2)=\varphi(x)^2\quad\mbox{and}\quad \varphi(xyx)=\varphi(x)\varphi(y)\varphi(x)
\end{equation}
for all $x,y\in R$. Writing $x\circ y$ for the Jordan product $xy+yx$ and 
linearizing the first equation, we see that $\varphi$ satisfies
\begin{equation} \label{j2}
    \varphi(x\circ y)=\varphi(x)\circ \varphi(y)
\end{equation}
for all $x,y\in R$. If $R'$ is $2$-torsion free (that is,  $2x= 0$ implies $x= 0$ in $R'$), then \eqref{j2} is actually equivalent to \eqref{j1}. This follows from
 $2x^2 =x\circ x$ and $2xyx = x\circ (x\circ y) - x^2 \circ y$.

A standard problem is to describe Jordan homomorphisms through their obvious examples, namely homomorphisms and antihomomorphisms.  This problem has a long history. Some of the early, classical results include Hua's theorem stating that a Jordan automorphism of a division ring is an automorphism or an antiautomorphism \cite{Hua},  Jacobson--Rickart's theorem   stating that a Jordan homomorphism of a  matrix ring is the sum of a homomorphism and an antihomomorphism \cite{JR}, Kadison's theorem stating that 
a Jordan $*$-homomorphism from a von Neumann algebra onto a C$^*$-algebra is  the direct sum of a $*$-homomorphism and a $*$-antihomomorphism \cite{K},
and Herstein's theorem stating that a Jordan homomorphism onto a prime ring (of characteristic not $2$) is  a homomorphism or an antihomomorphism \cite{Her}. These old results and their  generalizations are still relevant in the modern literature, primarily because Jordan homomorphisms naturally occur in various mathematical problems. In particular, the  solution of many, if not most of,  preserver problems is that the map in question is tightly connected to a (often  surjective) Jordan homomorphism. We list  a few  recent publications \cite{AEGV, BM, zpdbook, BFI, CFC, FV, F, GK, J, Mol, Sch} treating topics across mathematics,  in which Jordan homomorphisms  play a role.

A natural problem  is to find an appropriate generalization of the aforementioned Herstein's theorem on prime rings (that is, rings in which the product of two nonzero ideals is nonzero) to semiprime rings  (that is, rings without nonzero nilpotent ideals). This problem was first considered by Baxter and Martindale in \cite{BaxM}
and continued by the present author in \cite{B1}.
Since the direct product of two semiprime rings is again semiprime (but not prime), it is immediate that the  conclusion of  Herstein's theorem does not hold.  Moreover, an example in \cite{BaxM}, which  the authors attributed to Kaplansky, shows that not every Jordan homomorphism onto a semiprime ring is the sum of a homomorphism and an antihomomorphism. Generalizing the result of \cite{BaxM}, it was shown in 
\cite{B1} that a Jordan homomorphism $\varphi$ from a
ring $R$ onto a $2$-torsion free semiprime ring $R'$ is a  direct sum (in the sense of Definition \ref{d2} below) of a homomorphism and an antihomomorphism  from an essential ideal $I$ of $R$ onto an essential ideal $I'$ of $R'$. This result is in some sense optimal, but somewhat inconvenient for applications since  $I$ and $I'$ are  abstract,  unspecified essential ideals. 

Speaking philosophically, to understand a Jordan homomorphism $\varphi$, we must  understand its action on the Lie product $[x,y]=xy-yx$. This is simply because $2xy= x\circ y + [x,y]$ and we know how $\varphi$ acts on $x\circ y$ by the very definition. A natural ideal to study is thus the {\em commutator ideal} $K$ of $R$, that is, the ideal generated by all $[x,y]$, $x,y\in R$. Such a study was already  carried out  by the author in  
 \cite{B2}, under the assumption that $\varphi$ preserves tetrads
 $\{x_1,x_2,x_3,x_4\}=x_1x_2x_3x_4+x_4x_3x_2x_1$. In the present paper, we will avoid this assumption. The overlap between the two  papers is therefore  small.

The main idea of this paper is to introduce  the following notion. Given a Jordan homomorphism $\varphi:R\to R'$ and  denoting the subring of $R'$ generated by the image of $\varphi$ by $R_\varphi'$,  we say that $\varphi$
  is  {\em splittable} if the ideal of  $R_\varphi'$ generated by  $\{\varphi(xy)-\varphi(x)\varphi(y)\,|\, x,y\in R\}$ has trivial intersection with the ideal of 
  $R_\varphi'$ generated by 
 $\{\varphi(xy)-\varphi(y)\varphi(x)\,|\, x,y\in R\}$.
We provide several examples and show that such a Jordan homomorphism is the sum of a homomorphism and an antihomomorphism on the commutator ideal. This result is applied to the study of surjective Jordan homomorphisms onto semiprime rings as well as 
general Jordan homomorphisms to reduced rings 
(that is, rings without nonzero nilpotent elements).  We will in particular consider generalizations and connections with the classical results mentioned in the second  
paragraph,  thereby showing that the concept of a splittable Jordan homomorphism enables a unified approach to seemingly different problems.

The paper is organized as follows. In Section \ref{s2}, we introduce the notions of the sum and the  direct sum of a homomorphism and an antihomomorphism on an ideal,
and 
give two examples  that shed light on them.
 The central part of the paper is 
Section \ref{s4} where we define and study splittable Jordan homomorphisms.
In Section \ref{s5}, we consider
a Jordan homomorphism $\varphi$ that is the  sum of a homomorphism and an antihomomorphism on the commutator ideal $K$.  The question that we study is where   does $\varphi$ send elements from $K$.

\section{Sums and direct sums of homomorphisms and antihomomorphisms} \label{s2}

Throughout the paper,
unless stated otherwise, $R$ and $R'$ are assumed to be arbitrary associative  rings (possibly non-unital).
If $\varphi_1:R\to R'$ is a homomorphism 
and  $\varphi_2:R\to R'$ is an antihomomorphism, then 
$$\varphi=\varphi_1+\varphi_2$$
is a Jordan homomorphism, provided that
\begin{equation}\label{pp}
    \varphi_1(R) \varphi_2(R)=\varphi_2(R) \varphi_1(R)=\{0\}.
\end{equation} The simplest and the most natural situation where this is fulfilled is the following.

\begin{example} \label{example1}
Let $R'=R_1'\times R_2'$, let $\alpha_1:R\to R_1'$ be a homomorphism,
and let $\alpha_2:R\to R_2'$ be an antihomomorphism.
 Then $\varphi:R\to R'$, 
$$\varphi(x) 
= (\alpha_1(x),\alpha_2(x)),$$ is a Jordan homomorphism. Note that $\varphi$ is the  sum of
the homomorphism $\varphi_1(x)=(\alpha_1(x),0)$ and 
the antihomomorphism $\varphi_2(x)=(0,\alpha_2(x))$. 
\end{example}

Observe that, in this example, $\varphi(R)$ is not  closed under multiplication if $\alpha_1$ and $\alpha_2$ are injective and $R$ is noncommutative.  When one is interested    in surjective Jordan homomorphisms,
the following sub-example is more relevant.

\begin{example} \label{example2}
Let $R=R_1\times R_2$ and $R'=R_1'\times R_2'$, let $\beta_1:R_1\to R_1'$ be a homomorphism,
and let $\beta_2:R_2\to R_2'$ be an antihomomorphism.
 Then the Jordan homomorphism $\varphi:R\to R'$, 
$$\varphi(x_1,x_2)  
= (\beta_1(x_1),\beta_2(x_2)),$$ is the sum of
the homomorphism $\varphi_1(x_1,x_2)=(\beta_1(x_1),0)$ and 
the antihomomorphism $\varphi_2(x_1,x_2)=(0,\beta_2(x_2))$. 
\end{example}

It is natural to call $\varphi$ from this example the {\em direct sum} of a homomorphism  and an antihomomorphism. Observe that 
$\varphi(R)= (\beta_1(R_1),\beta_2(R_2))$
is a subring of $R'$. Moreover, if $\beta_1$ and $\beta_2$ are surjective, then so is $\varphi$.

As we mentioned in the introduction, and will also show in Examples \ref{mexa0} and \ref{mexa} below,  one can construct Jordan homomorphisms of relatively simple rings that are not sums of  homomorphisms and  antihomomorphisms. We will therefore consider their restrictions to the commutator ideal. The next definition, however, concerns
an arbitrary ideal. First, we need a notation: 
by $R_\varphi '$
we denote the subring 
of $R'$ generated by the image of the Jordan homomorphism $\varphi:R\to R'$.

\begin{definition}\label{d1}
 A  Jordan homomorphism   $\varphi:R\to R'$  is the 
   {\em sum of a homomorphism and an antihomomorphism on the ideal $I$} of $R$
 if there exist a homomorphism
    $\varphi_1:I\to R_\varphi'$ and an antihomomorphism  $\varphi_2:I\to R_\varphi'$  such that:
    \begin{enumerate}
        \item[(a)] $\left.\varphi\right|_{I}=\varphi_1+\varphi_2$,
        \item[(b)]  $R_\varphi '$ has  ideals $J_1$ and $J_2$  such that $\varphi_1(I)\subseteq J_1$, 
        $\varphi_2(I)\subseteq J_2$, 
 and    $J_1\cap J_2=\{0\}$,
        \item[(c)] $\varphi_1(ux)=\varphi_1(u)\varphi(x)$  and  $\varphi_1(xu)=\varphi(x)\varphi_1(u)$ 
 for all  
        $u\in I$ and $x\in R$,
        \item[(d)] $\varphi_2(ux)=\varphi(x)\varphi_2(u)$  and $\varphi_2(xu)=\varphi_2(u)\varphi(x)$ 
 for all  
        $u\in I$ and $x\in R$.
      \end{enumerate}  
\end{definition}

\begin{remark} \label{sar}     In the setup of Definition \ref{d1}, the following also hold:\begin{enumerate}
        \item[1.] 
If $\varphi$ is surjective, then it follows from (c) and (d) that 
$\varphi_1(I)$ and $\varphi_2(I)$
    are ideals of $R_\varphi '=R'$ (so we may take $J_1=\varphi_1(I)$ and $J_2=\varphi_2(I)$).
    \item[2.] In the basic case where $I=R$, (c) and (d) follow from other conditions (indeed,
    write $\varphi(x)=\varphi_1(x) + \varphi_2(x)$ and use \eqref{pp}). Definition \ref{d1} is therefore, in this case, just an interpretation of  Example \ref{example1}.
  \qedhere  \end{enumerate}
    \end{remark}

 The next definition corresponds  to  Example \ref{example2}.

\begin{definition}\label{d2}   A  Jordan homomorphism   $\varphi:R\to R'$  is the 
   {\em direct sum of a homomorphism and an antihomomorphism from the ideal $I$ of $R$ onto the ideal $J$ of $R'_\varphi$} if there exist 
   ideals $I_1$ and $I_2$ of $R$  such that:
    \begin{enumerate}
    \item[(a)] $I=I_1 + I_2$ and $I_1\cap I_2 = \ker \varphi\cap I$,
         \item[(b)] $\varphi(I_1)$ and $\varphi(I_2)$ are ideals of $R_\varphi '$ and
        $J=\varphi(I_1)\oplus \varphi(I_2)$,
        \item[(c)] $\varphi(u_1x)=\varphi(u_1)\varphi(x)$  for all 
        $u_1\in I_1$ and $x\in R$,
        \item[(d)]  $\varphi(u_2x)=\varphi(x)\varphi(u_2)$  for all 
        $u_2\in I_2$ and $x\in R$.
    \end{enumerate}
    \end{definition}


    \begin{remark}\label{re5}
        In the setup of Definition \ref{d2}, the following also hold:
        \begin{enumerate}
        \item[1.]  $\left.\varphi\right|_{I_1}$ is a homomorphism and $\left.\varphi\right|_{I_2}$ is an antihomomorphism.
      \item[2.] For all $u_1\in I_1$, $u_2\in I_2$, and $x\in R$,   \begin{equation*}\label{abc} \varphi(xu_1)=\varphi(x)\varphi(u_1) \quad\mbox{and}\quad \varphi(xu_2)=\varphi(u_2)\varphi(x)
      \end{equation*}
     (since $\varphi$ is a Jordan homomorphism).     
      \item[3.] $\varphi(I)=\varphi(I_1)\oplus \varphi(I_2)=J$.  Note also that if $\varphi$ is surjective, then the condition that $\varphi(I_1)$ and $\varphi(I_2)$
      are ideals follows from (c), (d), and the preceding remark.

     \item[4.] 
     $\varphi$ is the sum  of a homomorphism 
and an antihomomorphism  on $I$. They are defined as follows: $$\varphi_1(u_1+u_2)=\varphi(u_1)\quad\mbox{and}\quad \varphi_2(u_1+u_2)=\varphi(u_2)$$ for all $u_1\in I_1$ and $u_2\in I_2$. Indeed, observe that
$\varphi_1$ and $\varphi_2$ are well-defined,
$\varphi_1$ is a homomorphism,
$\varphi_2$ is an antihomomorphism, and all 
conditions of Definition \ref{d1} are fulfilled.
We also remark that $\left.\varphi_1\right|_{I_1}=\left.\varphi\right|_{I_1} $ and $\varphi_1(I_2)=\{0\}$,
$\left.\varphi_2\right|_{I_2}=\left.\varphi\right|_{I_2} $ and $\varphi_2(I_1)=\{0\}$, and $\varphi_1(I) =\varphi(I_1)$ and
        $\varphi_2(I) =\varphi(I_2)$.
      \qedhere  \end{enumerate}
    \end{remark}

We now see that
Definition \ref{d1} covers Definition \ref{d2},  that is, 
a "direct sum" is indeed a "sum". When is the converse true? We  answer this question for surjective Jordan homomorphisms in the following  lemma,  a variation of \cite[Lemma 2.1]{B2}.

\begin{lemma}\label{ll}
 Suppose a  surjective Jordan homomorphism $\varphi:R\to R'$
 is the sum of a homomorphism $\varphi_1$ and an antihomomorphism $\varphi_2$ on the ideal
    $I$. If  $\varphi_1(I)\subseteq \varphi(I)$, then $\varphi$ is the direct sum of a homomorphism and an antihomomorphism from the ideal $I$ onto the ideal $\varphi(I)$.
\end{lemma}

\begin{proof} 
    Let $I_1=\ker\varphi_2$ and $I_2=\ker\varphi_1$. Note that $I_1$ and $I_2$ are ideals of $R$ (not only  of $I$).
    Take $u\in I$. Then $\varphi_1(u)\in \varphi(I)$ by our assumption, so 
    there is a $v\in I$ such that $\varphi_1(u)=\varphi(v)$. On the other hand, $\varphi(v)=\varphi_1(v) + \varphi_2(v)$. Hence,
    $\varphi_1(u-v)= \varphi_2(v)$, which implies that 
    $u-v\in I_2$ and $v\in I_1$. Thus,
$u= v + (u-v)\in I_1+I_2$, proving that
$I=I_1+I_2$. 
Clearly, $I_1\cap I_2 = \ker\varphi\cap I$.
It is straightforward to check that 
$\varphi(I_i)$ is equal to $\varphi_i(I)$, $i=1,2$, and is therefore an ideal of $R'$ (see Remark \ref{sar}), as well as that 
conditions (c) and (d)  of Definition \ref{d2} are  fulfilled.
\end{proof}

The difference between "sum" and "direct sum" is thus subtle, but important.  Lemma \ref{ll} implies that the two  coincide in the basic case where  $I=R$ and $\varphi$ is surjective.

We will now give two examples of Jordan homomorphisms that are  sums of homomorphisms and antihomomorphisms on the commutator ideal, but not on the whole ring. They were inspired by the example from \cite{BaxM} we mentioned in the introduction.

    \begin{example}\label{mexa0}
Let $R$ be a unital ring satisfying the 
 following two conditions:
\begin{enumerate}
    \item[(a)] the commutator ideal $K$ of $R$ is not contained in a proper ideal that is a direct summand (i.e., if $e\in R$ is a central idempotent such that $K\subseteq eR$, then $e=1$),
    \item[(b)] $R$ contains a proper ideal $T$ 
    and a subring $A$  containing $1$ and contained in the center $Z(R)$ 
    such that  $R=T\oplus A$  (the additive group direct sum).   
\end{enumerate}
We will show that then there exist a ring $R'$ and a Jordan homomorphism $\varphi:R\to R'$ which is  the sum of a homomorphism and an antihomomorphism on $T$, but not on the whole $R$. Since $T$ obviously contains $K$,
$\varphi$ is, in particular, the sum of a homomorphism and an antihomomorphism on $K$.

We remark that a simple  example of such a ring $R$ is the ring obtained by adjoining a unity to 
any noncommutative ring $T$ without  nonzero central idempotents (so, in particular, without unity).

We now define $R'$ and $\varphi$. Let  $\psi:R\to T$ and
$\pi:R\to A$  be the projections, i.e., 
$\psi(t+a)=t$ and $\pi(t+a)=a$   for all $a\in A$ and $t\in T$. Observe that 
$\pi$ is a homomorphism,  $\pi(1)=1$,  $\psi(1)=0$, and $x=\psi(x)+\pi(x)$ for every $x\in R$.
Write $T^{\rm o}$ for the opposite ring of $T$. Note that  $R'=R\times T^{\rm o}$ becomes a   ring under the componentwise addition and  multiplication given by
\begin{equation}\label{lab}
    (x,t)(x',t') = (xx', \pi(x)t' +\pi(x')t +tt').
\end{equation} 
Moreover, $R'$ is unital with unity  $1'=(1,0)$. Define
$\varphi:R\to R'$ by
$$\varphi(x)= (x,\psi(x)).$$
One easily checks that $\varphi$ is a Jordan homomorphism and that $\varphi$ is the sum of the homomorphism $t\mapsto (t,0)$ and the antihomomorphism $t\mapsto (0,t)$ on $T$ (the ideals from Definition \ref{d1} are  $J_1 = T\times \{0\}$ and 
$J_2 = \{0\}\times T^{\rm o}$).

Assume now that there exist a homomorphism $\varphi_1:R\to R'$ and an antihomomorphism $\varphi_2:R\to R'$
such that $\varphi$ is the sum of $\varphi_1$ and $\varphi_2$ (on the whole $R$). Write $e_1=\varphi_1(1)$
and $e_2=\varphi_2(1)$. Clearly, $e_1$ and $e_2$ are idempotents and  $e_1 + e_2 =\varphi(1)=1'$.  From \begin{equation}\label{alsoh}
    e_1\varphi_1(x)=\varphi_1(x) = \varphi_1(x)e_1\quad\mbox{and}\quad e_2\varphi_2(x)=\varphi_2(x) = \varphi_2(x)e_2
\end{equation}  we see that  $e_1$ and $e_2$ commute with every  $\varphi(x)=\varphi_1(x)+\varphi_2(x)$, $x\in R$.
    Moreover,
\begin{equation}
    \label{eena} \varphi_1(x)=e_1 \varphi(x)\quad\mbox{and}\quad\varphi_2(x)=e_2 \varphi(x).
\end{equation}
    Let $f_1\in R$ and
    $t_1\in T$ be such that $e_1=(f_1,t_1)$. From $e_1^2=e_1$ we infer that
    $f_1^2=f_1$ and \begin{equation}\label{beco}
        2\pi(f_1)t_1+t_1^2=t_1,
    \end{equation}
    and from $e_1\varphi(x)=\varphi(x)e_1$ we infer that $f_1x=xf_1$, $x\in R$. Thus, $f_1$ is a central idempotent of  $R$.
     Next, the condition that $\varphi_2$ is an antihomomorphism can be in view of  \eqref{eena}  written as
    \begin{equation}
\label{sinn}e_2\varphi(xy)=e_2\varphi(y)\varphi(x)
    \end{equation}
    for all $x,y\in R$. Since $e_2=1'-e_1= (1-f_1,-t_1)$,
    this implies $(1-f_1)[x,y]=0$
for all $x,y\in R$. That is, the commutator ideal $K$ is contained in $f_1R$, so the assumption (a) implies  $f_1=1$. Therefore, \eqref{beco} becomes
$t_1^2 = -t_1$, i.e., $-t_1$ is an idempotent. Since $e_1=(1,t_1)$ commutes with every $\varphi(x)$,
$t_1$ commutes with every element in $T$, and hence with every element in $R$. Therefore,
$-t_1$ is a central idempotent of $R$.
Choosing $x$ and $y$ in \eqref{sinn} to be from $T$, we obtain  
$(-t_1)[x,y] = 0$ for all $x,y\in T$, and hence also for all $x,y\in R$. By (a), $-t_1=0$. Therefore, $e_2=0$ and $\varphi=\varphi_1$ is a homomorphism.
However, from $\varphi(tt') = \varphi(t)\varphi(t')$ with $t,t'\in T$ it  follows that $tt'=t't$ for all $t\in T$, implying that $R$ is commutative. This, of course, contradicts (a). 
\end{example}

The next example is similar, but considers a Jordan automorphism. In particular, it shows
that Jordan automorphisms of reduced rings are not always sums of homomorphisms and antihomomorphisms (on the whole ring).

\begin{example}\label{mexa}
    Let $S$  and $T$ be noncommutative prime rings, with $S$ unital and $T$ not. 
    Assume 
    there exists a homomorphism $\pi$ from $S$ onto a commutative unital ring 
    $A$ such that $T$ is an algebra over $A$. Assume further that there exists an $A$-linear antiautomorphism $\psi$ of $T$. 
 Note that $R=S\times T$ becomes a unital semiprime ring under the componentwise addition and  multiplication given by the same rule \eqref{lab} as in the preceding example, where we now take $x$ from $S$ and $t$ from 
 $T$.
 Moreover,
    if both $S$ and $T$ are domains, then  $R$ is reduced. 

    Before proceeding, we remark that concrete examples of the situation discussed 
    can be easily found.  For instance, one can take $S$ to be a free noncommutative  algebra (with at least two generators) over a field $A$, $T$ to be the ideal of $S$ consisting of polynomials with zero constant term (so $T$ is an algebra over $S/T\cong A$), and $\psi$
    to be the involution that fixes the generators (compare \cite{BaxM}). Since $S$ and $T$ are domains, $R$ is a reduced ring.

    Observe that
    $$\varphi(s,t)=(s,\psi(t))$$
    defines a Jordan automorphism of $R$ which is neither an automorphism nor an antiautomorphism.
We claim that, moreover, $\varphi$ is not a sum of a homomorphism and an antihomomorphism. Assume, on the contrary, that there exist a homomorphism $\varphi_1:R\to R$ and      
    an antihomomorphism $\varphi_2:R\to R$
    such that  $\varphi=\varphi_1+\varphi_2$. Since
    $\varphi(1)=1$, $e_1=\varphi_1(1)$ and
    $e_2=\varphi_2(1)$ are orthogonal idempotents whose sum is $1$. Note also that they satisfy \eqref{alsoh}. 
    Since $\varphi$ is surjective, it follows  that $e_1$ and 
    $e_2$ are central idempotents. However,
    we have assumed that 
$S$ and $T$ are prime and  $T$ is not unital, from which we infer that $R$ has no 
    nontrivial central idempotents. This  leads to a contradiction  that 
    $\varphi=\varphi_1$ or $\varphi=\varphi_2$.

    Let $K_S$ (resp. $K_T$) denote the commutator ideal of $S$ (resp. $T$). 
It is easy to see that the commutator ideal of $R$
is equal to $K=K_S\times K_T$. Note that
$K_S$ lies in the kernel of $\pi$, which implies that $I_1=K_S\times \{0\}$ is an ideal of $R$. Clearly, so is $I_2=\{0\}\times K_T$, and $I_1\oplus I_2=K$. It is easy to check that  $I_1$ and $I_2$ satisfy  conditions of Definition \ref{d2}, showing that
 $\varphi$ is the direct sum of a homomorphism and an antihomomorphism from
 $K$ onto $K$.

Incidentally, if $T$ had a unity $1_T$, then $\varphi$ would be the direct sum of a homomorphism and an antihomomorphism on the whole $R$. Indeed, $e=(0,1_T)$ would  be a central idempotent  and $\varphi$ would be an automorphism of $(1-e)R = \{(s,-\pi(s)1_T)\,|\,s\in S\}$ and an antiautomorphism of $eR = \{0\}\times T$.
\end{example}

Of course, Example \ref{mexa} tells us more than  Example \ref{mexa0}. The point of the latter, however, is that the conditions on $R$ are relatively mild, showing that nonstandard Jordan homomorphisms on rings that do not coincide with their commutator ideals often occur. 

\section{Splittable Jordan homomorphisms}\label{s4}

For a Jordan homomorphism   
$\varphi:R\to R'$,
we write
      \begin{equation}\label{nota}
      \{x,y\}=\varphi(xy)-\varphi(x)\varphi(y)\quad\mbox{and}\quad \langle x,y\rangle=\varphi(xy)-\varphi(y)\varphi(x).\end{equation}
The maps $(x,y)\mapsto     \{x,y\}$ and 
 $(x,y)\mapsto     \langle x,y\rangle$ are biadditive and skew-symmetric, i.e., they satisfy
      $$\{x,y\} =     -\{y,x\}\quad\mbox{and}\quad
           \langle x,y\rangle =     -\langle y,x\rangle.$$
           A less obvious, but well-known 
           property  is
    \begin{equation}\label{jh}
      \{x,y\} \langle x,y\rangle=\langle x,y\rangle    \{x,y\} = 0\end{equation}
      for all $x,y\in R$ (see, for example, \cite[Lemma 4]{Her}).

      Recall that $R_\varphi '$ denotes the subring of $R'$ generated by $\varphi(R)$.     
By $V_\varphi$  we denote the ideal of $R_\varphi '$ generated by all 
$\{x,y\}$, $x,y\in R$, and  by $W_\varphi$  the ideal of $R_\varphi '$ generated by all   $\langle x,y\rangle$, $x,y\in R$. Clearly, $\varphi$ is a homomorphism if and only if $V_\varphi=\{0\}$,
 and $\varphi$ is an antihomomorphism if and only if $W_\varphi=\{0\}$.

\begin{definition}
    A Jordan homomorphism $\varphi:R\to R'$ is
    {\em splittable} if $V_\varphi\cap W_\varphi=\{0\}$.
\end{definition}

We will see that this notion is connected with the notion of the sum of a homomorphism and an antihomomorphism.
In one direction this is clear.

\begin{lemma}\label{easy}
    If a Jordan homomorphism $\varphi:R\to R'$  is the sum of a homomorphism and an antihomomorphism (on the whole $R$), then $\varphi$ is splittable.
\end{lemma}

\begin{proof}
Using the notation from Definition \ref{d1} (for $I=R$), we have
    $$\{x,y\} = \varphi_1(xy) + \varphi_2(xy) - (\varphi_1(x)+\varphi_2(x))(\varphi_1(y)+\varphi_2(y)).$$
    Since \eqref{pp} holds, it follows that
    \begin{align*}\{x,y\} = \varphi_1(xy) + \varphi_2(xy) - \varphi_1(x)\varphi_1(y) -\varphi_2(x)\varphi_2(y) = \varphi_2([x,y]) \in J_2.\end{align*}
    Consequently, $V_\varphi\subseteq J_2$.
    Similarly we see that $W_\varphi\subseteq J_1$. Since $J_1\cap J_2=\{0\}$, we also have
    $V_\varphi\cap W_\varphi=\{0\}$.
\end{proof}

The converse of Lemma \ref{easy} does not hold in general.

\begin{example} We claim that Jordan homomorphisms from Examples \ref{mexa0} and
\ref{mexa} are splittable.
Indeed, 
the  Jordan  homomorphism  $\varphi$  from Example \ref{mexa0} satisfies $V_\varphi\subseteq  \{0\}\times T^{\rm o}$ and
$W_\varphi\subseteq  T\times \{0\}$, and hence 
$V_\varphi\cap W_\varphi=\{0\}$. Similarly,
 the  Jordan  homomorphism  $\varphi$  from Example \ref{mexa} satisfies   $V_\varphi\subseteq  \{0\}\times T$ and $W_\varphi\subseteq  S\times \{0\}$, and hence 
$V_\varphi\cap W_\varphi=\{0\}$.
\end{example}

 Our initial motivation for introducing splittable Jordan homomorphisms was \cite[Corollary 2.2]{B1}
 which states that if $R'$ is a $2$-torsion free semiprime ring, then
 every  surjective Jordan homomorphism
$\varphi:R\to R'$ satisfies
$V_\varphi  W_\varphi=\{0\}$.  This, however, is equivalent to 
$V_\varphi \cap W_\varphi=\{0\}$ (as $R'$ is semiprime and  $V_\varphi  W_\varphi=\{0\}$ implies $(V_\varphi \cap W_\varphi)^2=\{0\}$). Thus,
the following is true.

  \begin{proposition}    \cite{B1}\label{rabim} 
  If the ring $R'$ is $2$-torsion free and semiprime, then every surjective Jordan homomorphism
$\varphi:R\to R'$ is splittable.
\end{proposition}

Since the intersection of two nonzero ideals of a prime ring is nonzero,
 Proposition \ref{rabim} covers Herstein's theorem from \cite{Her}.

Our next goal is to prove that general,  not necessarily surjective Jordan homomorphisms to reduced rings are splittable. We need a lemma first.

\begin{lemma}\label{rfg} Let $f,g:R\to R'$
be   additive maps satisfying $f(x)g(x)=g(x)f(x)=0$ for all $x\in R$. If $R'$ is a reduced ring, then  
$f(x)g(z)=g(z)f(x)=0$ for all $x,z\in R$. \end{lemma}

\begin{proof} 
   Linearizing $f(x)g(x)=0$ we obtain
     $f(x)g(z)+f(z)g(x)=0$ for all $x,z\in R$. Multiplying this equation from the right by $f(x)g(z)$ gives       $(f(x)g(z))^2=0$. As $R'$ is reduced,    $f(x)g(z)=0$ follows. This yields $(g(z)f(x))^2=0$ and hence $g(z)f(x)=0$. 
\end{proof}

\begin{proposition}    \label{rabimr} 
If the ring $R'$ is reduced, then every Jordan homomorphism
$\varphi:R\to R'$ is splittable.\end{proposition}

\begin{proof}  Fixing $y$ in \eqref{jh}, we obtain from Lemma \ref{rfg} that 
  \begin{equation}\label{once}
      \{x,y\} \langle z,y\rangle=\langle z,y\rangle    \{x,y\} = 0
  \end{equation} 
    for all $x,y,z\in R$. Now, fixing $x$ and $z$ in \eqref{once} and using Lemma \ref{rfg} again, we obtain     $$\{x,y\} \langle z,w\rangle=\langle z,w\rangle    \{x,y\} = 0$$
    for all $x,y,z,w\in R$. Hence, 
    the element $\{x,y\} t\langle z,w\rangle$
    has  square zero for any $t\in R'$, and is therefore equal to $0$ as $R'$ is reduced. This implies that $V_\varphi W_\varphi=\{0\}$. Since $R'$ is reduced,   $V_\varphi \cap W_\varphi=\{0\}$ follows.
\end{proof}

This proposition generalizes Hua's theorem on Jordan automorphisms of division rings \cite{Hua}, as well as the Jacobson--Rickart generalization   \cite[Theorem 2]{JR} which states that a Jordan homomorphism from $R$ to a domain $R'$ is either a homomorphism or an antihomomorphism. This is because a subring of a domain cannot contain nonzero ideals with trivial intersection. 
 
Unlike in  Propositions \ref{rabim} and \ref{rabimr},  in the next two propositions we will impose  conditions  on $R$ rather than on $R'$. 

By $M_n(S)$ we denote the ring 
of $n\times n$ matrices over 
the ring $S$. The following result
was proved by Jacobson and Rickart
\cite[Theorem 7]{JR}. We state it in a way that 
 includes  Lemma \ref{easy}.

 \begin{proposition} \cite{JR}   \label{rabimjr} 
If $R=M_n(S)$ where $n\ge 2$ and $S$ is a unital ring,   then every Jordan homomorphism
$\varphi:R\to R'$ is the sum of a homomorphism and an antihomomorphism (on the whole $R$). In particular, $\varphi$ is splittable.\end{proposition}

The next proposition considers unital  rings in which the commutator of two elements is equal to $1$. The most well-known examples are Weyl algebras. It is easy to construct other examples.
For instance, if $A$ is any algebra having such a pair of elements, then $A\otimes B$, where $B$ is any unital algebra, has them too. The class of rings having such a pair of elements is closed under direct products and homomorphic images.

In the proof, we will use the standard fact that every 
Jordan homomorphism $\varphi:R\to R'$ satisfies
 \begin{equation}
     \label{ch} \varphi([[x,y],z])=
     [[\varphi(x),\varphi(y)],\varphi(z)]
 \end{equation}  
for all $x,y,z\in R$. This follows from  $x\circ (z\circ y)- (x\circ z)\circ y = [[x,y],z]$. 

 \begin{proposition}\label{rabimjr2} 
If a $2$-torsion free unital ring $R$ contains elements $a$ and $b$ such that
 $[a,b]=1$, then every Jordan homomorphism
$\varphi:R\to R'$ is the sum of a homomorphism and an antihomomorphism (on the whole $R$). In particular, $\varphi$ is splittable.\end{proposition}

\begin{proof} Write
$e=\{a,b\}$, $f=\langle a,b\rangle$, and $1'=\varphi(1)$. It is easy to see that $1'$ is the unity of $R_\varphi '$ \cite[Corollary 3]{JR}.
We have 
$$e+ f = 2\varphi(ab)- \varphi(a)\circ \varphi(b) = 2\varphi(ab)- \varphi(a\circ b)=\varphi([a,b])= 1'.$$
Since $ef=0$ by \eqref{jh}, it follows that $e$ and $f$ are idempotents.   Next, from \eqref{ch} we obtain $[[\varphi(a),\varphi(b)],\varphi(x)]=0$
for every $x\in R$. That is,
$[\varphi(a)\varphi(b),\varphi(x)]= [\varphi(b)\varphi(a),\varphi(x)]$, and
hence $$[e,\varphi(x)]= [\{a,b\},\varphi(x)] = [\langle a,b\rangle,\varphi(x)]= [f,\varphi(x)].$$
 However, $f=1'-e$, so this gives
 $2[e,\varphi(x)]=0$. As $R'$ is $2$-torsion free, this means that $e$ lies in the center of $R_\varphi '$. Thus,
 $e$ and $f=1'-e$ are orthogonal central idempotents in $R_\varphi'$.

 Linearizing the first equation in
 \eqref{jh}, we obtain
 \begin{equation}
     \label{jh1}
    \{x,z\} \langle x,y\rangle+ \{x,y\} \langle x,z\rangle    =0
 \end{equation}
 for all $x,y,z\in R$.
 Further, the linearization of \eqref{jh1} gives
 \begin{equation}
     \label{jh2}
    \{x,z\} \langle w,y\rangle+
    \{w,z\} \langle x,y\rangle+  
    \{x,y\} \langle w,z\rangle +  
    \{w,y\} \langle x,z\rangle =0
 \end{equation}
 for all $x,y,z,w\in R$.
 Setting 
 $x=a$ and $z=b$ in \eqref{jh1}, we obtain 
 $e\langle a,y\rangle + \{a,y\}f=0$. Since $e$ and $f$ are orthogonal central idempotents, this gives
  $e\langle a,y\rangle =0$ for all $y\in R$. Similarly, $e\langle b,y\rangle =0$ for all $y\in R$.
 Therefore, writing $w=a$ and $z=b$ in \eqref{jh2} and multiplying the obtained equation by $e$, we obtain
 $e\langle x,y\rangle=0$ for all $x,y\in R$.
  Similarly we derive
 $f\{x,y\}=0$ for all $x,y\in R$. Consequently, $\varphi_1:R\to R_\varphi'$, $\varphi_1(x)=f\varphi(x)$, is a homomorphism, and $\varphi_2:R\to R_\varphi'$, $\varphi_2(x)=e\varphi(x)$, is an antihomomorphism.
 Clearly, $\varphi=\varphi_1+\varphi_2$ and all conditions of Definition \ref{d1} are fulfilled (for $I=R$).
\end{proof}

Corollary \ref{ts} below will show that  the additional information in the last two propositions that $\varphi$ is splittable is actually equivalent to the main statement.

The following example, a reinterpretation of \cite[Example 7.21]{zpdbook}, shows that not every 
 Jordan automorphism is splittable. 

\begin{example}\label{gras}
Let $R$ be the Grassmann algebra with two generators $a$ and $b$ over a field $F$ of characteristic not $2$.  Thus, $a$ and $b$ satisfy $a^2=b^2=ab+ba=0$ and $R$ is a $4$-dimensional $F$-algebra with basis $\{1,a,b,ab\}$.
Observe that every element in $U=$ span$\{a,b,ab\}$ has square zero. Therefore, every bijective linear map
$\varphi:R\to R$ that maps $1$ to $1$ and 
$U$ to $U$ is a Jordan automorphism. Observe that the product of any two elements from $U$ lies in $K=Fab$ (which, incidentally, is the commutator ideal of $R$). Therefore, if we choose $\varphi$ so that $\varphi(ab)\notin K$, then $\varphi$ is neither a homomorphism nor an antihomomorphism, meaning that $V_\varphi\ne \{0\}$
and $W_\varphi\ne \{0\}$. Since every nonzero ideal of $R$ contains $K$, we must have $K\subseteq V_\varphi\cap W_\varphi$. Therefore, $\varphi$ is not splittable.
\end{example}

Example \ref{gras} is very simple and  will be used   later for other purposes. There are other examples of Jordan homomorphisms that are not splittable. Note that if one of the rings $R$ and $R'$ is commutative, then  $V_\varphi=W_\varphi$ holds for every Jordan homomorphism $\varphi:R\to R'$, so $\varphi$ is splittable if and only if $\varphi$ is a homomorphism (and simultaneously an antihomomorphism). Jacobson and Rickart 
showed that there exist 
Jordan homomorphisms from
a commutative  ring $R$ to a
 (noncommutative) ring $R'$ that are not   homomorphisms, and hence are not  splittable Jordan homomorphisms \cite[Examples 1 and 2]{JR}. In the second example, $R$ is a field.

Each of nostandard Jordan homomorphisms
from  Examples \ref{mexa0}, \ref{mexa}, and  \ref{gras}, as well as from \cite[Examples 1 and 2]{JR}, is the sum  of  a homomorphism and an antihomomorphism on the commutator ideal. Let us finally give an example where this is not the case.

\begin{example}
    Again we let $R$  be a Grassmann algebra over a field $F$ of characteristic not $2$, but this time with three generators $a$, $b$, and $c$.   Thus, $a$, $b$, $c$ have square zero and pairwise anticommute, and $R$ is an $8$-dimensional $F$-algebra with basis 
    $\{1,a,b,c,ab,ac,bc,abc\}$. A straightforward verification shows that the linear map $\varphi:R\to R$ given by
    $$ \varphi(1)=1,\,\,\, \varphi(a)=bc,\,\,\,\varphi(b)=ac,\,\,\,\varphi(c)=ab,$$
       $$ \varphi(ab)=c,\,\,\, \varphi(ac)=b,\,\,\,\varphi(bc)=a,\,\,\,\varphi(abc)=abc,$$
       is a Jordan automorphism. 
       It is easy to see that 
       $V_\varphi=W_\varphi={\rm span}\{a,b,c,ab,ac,bc,abc\}$, so $\varphi$ is not splittable. 

       Observe that the commutator ideal $K$ of $R$ is equal to
       span$\{ab,ac,bc,abc\}$. Suppose
$\left.\varphi\right|_{K}=\varphi_1+\varphi_2$ for some homomorphism $\varphi_1:K\to R$ and some antihomomorphism $\varphi_2:K\to R$.
Since the product of any two elements in $K$ is $0$, we have $\varphi_i(k)\varphi_i(k')=0$
for all $k,k'\in K$, $i=1,2$. In particular,
$u_1=\varphi_1(ab)$, $u_2=\varphi_1(ac)$, $u_3=\varphi_1(bc)$ satisfy $u_iu_j=0$, $i,j=1,2,3$.
We claim that this  implies that 
 there exists an $x\in {\rm span}\{a,b,c\}$ such that
        $u_i\in {\rm span}\{x, ab,ac,bc, abc\}$, $i=1,2,3$.
     Indeed, this follows easily from the 
     observation that the product of two elements in span$\{a,b,c\}$ is $0$ if and only if they are linearly dependent, along with the
     obvious fact that elements having square zero lie in ${\rm span}\{a,b,c,ab,ac,bc,abc\}$.   
Analogously,
there exists a $y\in {\rm span}\{a,b,c\}$ such that 
each of the elements $v_1=\varphi_2(ab)$, $u_v=\varphi_2(ac)$, $v_3=\varphi_2(bc)$
lies in ${\rm span}\{y, ab,ac,bc, abc\}$. However, since $$c=\varphi(ab)=u_1+v_1,\quad  b=\varphi(ac)=u_2+v_2,\quad  a=\varphi(bc)=u_3+v_3,$$ this leads to the contradiction that $a,b,c \in {\rm span}\{x,y, ab,ac,bc, abc\}$. Therefore,
$\varphi$
 is not the sum of a homomorphism and an antihomomorphism on $K$.        
\end{example}

The above propositions and examples justify the introduction of a  splittable Jordan homomorphism. Its definition, however, is technical and does not reveal much about its structure. This will be remedied in the next theorem. We record two  remarks before stating it.

\begin{remark}\label{rli}
    Recall that an additive subgroup $A$ of $R$ is called a {\em Lie ideal} if $[A,R]\subseteq A$. 
    From $xay = x[a,y] + xya$, where
    $a \in A$ and $x,y\in R$, we see that
  the left ideal generated by $A$ is already a two-sided ideal. Thus, for example, the commutator ideal
 $K$  of $R$ is the left ideal generated by $[R,R]$,  so it consists of sums of elements of the form $x[y,z]$ and, since we are not assuming that our rings are unital, also  elements of the form $[y,z]$. Similarly, the ideal $L$ of $R$ generated by $[[R,R],R]$ consists of sums of elements of the form $w[x,[y,z]]$ and $[x,[y,z]]$.
 Similar observations of course apply for right ideals generated by Lie ideals.
\end{remark}

\begin{remark}
    Suppose a Jordan homomorphism $\varphi:R\to R'$ is the sum of a homomorphism $\varphi_1$ and an antihomomorphism $\varphi_2$ (on the whole $R$). Then 
    \begin{align*}
        \varphi_1([x,y])  =\varphi(xy) -\varphi_2(xy) - \varphi_1(yx)  =\varphi(xy) - \varphi_2(y)\varphi_2(x) - \varphi_1(y)\varphi_1(x). 
    \end{align*}
    Since $\varphi_1(R)\varphi_2(R)= \varphi_2(R)\varphi_1(R) = \{0\}$, it follows that
    \begin{equation}
        \label{hold} \varphi_1([x,y])= \langle x,y\rangle.
    \end{equation}
Similarly we see that
   \begin{equation}
        \label{hold2}   \varphi_2([x,y])= \{ x,y\}.
    \end{equation}
    These observations will not be needed in what follows, but they provide an insight into the proof of the following theorem, which is our central result.     
\end{remark}

\begin{theorem}\label{rmt}
    If a Jordan homomorphism $\varphi:R\to R'$ is  splittable, then  $\varphi$ is the sum of a homomorphism and an antihomomorphism on the commutator ideal $K$ of $R$.
\end{theorem}

\begin{proof}
First, a general remark. As mentioned in Remark \ref{rli}, the commutator ideal $K$ consists of sums of elements of the form $x[y,z]$ and
$[y,z]$.  However, for simplicity of exposition we will deal only with  $x[y,z]$  and ignore the commutators $[y,z]$. It should be obvious what changes need to be made to cover the general situation.

In the first step of the proof, we will show that if 
 $x_i,y_i,z_i\in  R$ are such that
\begin{equation}\label{za}
\sum_i x_i[y_i,z_i] =0,    
\end{equation}
then both
$$
 v=
\sum_i  \{y_i,z_i\} 
 \varphi (x_i)\,\,\,\mbox{and}\,\,\,
w=\sum_i \varphi(x_i)\langle y_i,z_i\rangle 
 $$
are equal to $0$.
We start by writing \eqref{za} as 
\begin{equation}
    \label{bef}\sum_i x_i\circ (y_iz_i)=
\sum_i x_iz_iy_i + y_iz_ix_i.
\end{equation}
Linearizing $\varphi(xyx)=\varphi(x)\varphi(y)\varphi(x)$, we see that $\varphi$ satisfies
 \begin{equation}
     \label{ch3} \varphi(xyz+zyx)=\varphi(x)\varphi(y)\varphi(z) + \varphi(z)\varphi(y)\varphi(x)\end{equation}
for all $x,y,z\in R$. Applying $\varphi$ to \eqref{bef}, we therefore obtain 
 \begin{align*}\sum_i \varphi(x_i)\circ \varphi(y_iz_i)= 
\sum_i\varphi(x_i)\varphi(z_i)\varphi(y_i) + \varphi(y_i)\varphi(z_i)\varphi(x_i). 
 \end{align*}
This can be read as $v=-w$. 
Note that  $v\in V_\varphi$ and $w\in W_\varphi$.
Since $ V_\varphi\cap W_\varphi=\{0\}$ by assumption, it follows that 
 $v=w=0$, as desired.

What we just proved shows that there are well-defined additive
maps $\varphi_1,\varphi_2:K\to R_\varphi '$ satisfying 
\begin{equation}\label{defi}
\varphi_1( x[y,z]) =   \varphi(x)\langle y, z\rangle \quad\mbox{and}\quad \varphi_2( x[y,z])=\{ y,z\}\varphi(x)    
\end{equation}for all $x,y,z\in R$. Our goal is to prove that 
 $\varphi_1$ and $\varphi_2$ satisfy all conditions of Definition \ref{d1}. We first observe that condition (b)  is fulfilled
 since $\varphi_1(K)\subseteq W_\varphi$, $\varphi_2(K)\subseteq V_\varphi$, and $V_\varphi\cap W_\varphi=\{0\}$.
 
In light of the convention mentioned in  the first paragraph of the proof, $\varphi_1$ satisfies \eqref{hold} and $\varphi_2$ satisfies \eqref{hold2}. 
 Let us show that
 \begin{equation} \label{bbc}\varphi_1([y,z]x)= \langle y, z\rangle\varphi(x)\quad\mbox{and}\quad \varphi_2([y,z]x)= \varphi(x)\{ y, z\}
 \end{equation}
 for all $x,y,z\in R$. 
 We have
\begin{align*}
    \varphi_1([y,z]x) = &\varphi_1([[y,z],x])+\varphi_1(x[y,z])\\ =& \langle[y,z],x\rangle + \varphi(x)\langle y, z\rangle\\
    =& \varphi([y,z]x) - \varphi(x)\varphi([y,z]) + \varphi(x)\varphi(yz)  - \varphi(x)\varphi(z)\varphi(y)\\
    =& \Bigl(\varphi(yzx+xzy) - \varphi(x)\varphi(z)\varphi(y)\Bigr) -\Bigr(\varphi(x\circ zy) - \varphi(x)\varphi(zy) \Bigr).
\end{align*}
Since  $\varphi$ is a Jordan homomorphism
 and since 
\eqref{ch3} holds, it follows that 
 \begin{align*}
    \varphi_1([y,z]x) = &\varphi(y)\varphi(z)\varphi(x)  -\varphi(zy)\varphi(x)  \\ = & -\langle z,y\rangle\varphi(x)= \langle y,z\rangle \varphi(x).
\end{align*}
We have thus proved the first equation in 
\eqref{bbc}. The proof of the second one is similar.

Next,  we  show that
 \begin{equation}
    \label{dva}
\varphi_1(kx) = \varphi_1(k) \varphi(x)\quad\mbox{and}\quad \varphi_1(xk) = \varphi(x) \varphi_1(k),
\end{equation}and
\begin{equation}
\label{tri}
\varphi_2(kx) = \varphi(x) \varphi_2(k)\quad\mbox{and}\quad  \varphi_2(xk)= \varphi_2(k) \varphi(x)
\end{equation}
for all $k\in K$ and $x\in R$. 
As $K$ is equal to the right ideal generated by $[A,A]$, we can write
 $k=\sum_i [y_i,z_i]x_i$ (so, again we are neglecting the commutators $[y_i,z_i]$ for simplicity). By \eqref{bbc},
$$\varphi_1(kx) = \sum_i \langle y_i,z_i\rangle 
 \varphi (x_ix).$$
 However, since
 $$\langle y_i,z_i\rangle \bigl(\varphi( x_ix) -\varphi( x_i)\varphi(x)\bigr)\in W_\varphi V_\varphi =\{0\}, $$
this gives
 $$\varphi_1(kx)=\sum_i  \langle y_i,z_i\rangle 
 \varphi (x_i)\varphi(x)=\varphi_1(k)\varphi(x),$$ as desired. The proof that $\varphi_1(xk) = \varphi(x) \varphi_1(k)$ is similar, just that we use \eqref{defi}  rather than \eqref{bbc}.  Thus, \eqref{dva} holds. The proof of \eqref{tri} is analogous.

We claim that
\begin{equation}\label{f12}
\left.\varphi\right|_{K}=\varphi_1+
\varphi_2. 
\end{equation}
Indeed, take $k=\sum_i x_i[y_i,z_i]\in K$. Using \eqref{ch3}, we  have
\begin{align*}
    \varphi_1(k)+\varphi_2(k)=& 
   \sum_i  \varphi(x_i)\langle y_i,z_i\rangle 
+ \{ y_i,z_i\} 
 \varphi (x_i)
    \\=&\sum_i \varphi(x_i)\circ \varphi(y_iz_i) 
   -\sum_i \bigl(\varphi(x_i)\varphi(z_i)\varphi(y_i) + \varphi(y_i)\varphi(z_i)\varphi(x_i)\bigr)
    \\=&\sum_i \varphi(x_i\circ y_iz_i) - \sum_i \varphi (x_iz_iy_i + y_iz_ix_i)\\
    =&\varphi\left(\sum_i x_i[y_i,z_i]\right)= \varphi(k).
\end{align*}

The only conditions from Definition \ref{d1} that remain to be verified are that
 $\varphi_1$ is a homomorphism and $\varphi_2$ is an antihomomorphism. Take $k,l\in K$. By \eqref{dva} and \eqref{f12}, 
$$\varphi_1(kl)= \varphi_1(k)\varphi(l) = \varphi_1(k)\varphi_1(l) + \varphi_1(k)\varphi_2(l).$$
However, $\varphi_1(k)\varphi_2(l)\in W_\varphi V_\varphi=\{0\}$, 
and hence $\varphi_1(kl)= \varphi_1(k)\varphi_1(l)$, i.e., $\varphi_1$ is a homomorphism. Similarly we see that $\varphi_2$ is an antihomomorphism.
\end{proof}

We remark that \cite[Theorem 4.2]{B2} gives essentially the same conclusion on a Jordan homomorphism $\varphi$ under the assumptions  that $\varphi$ preserves tetrads and $R_\varphi '$ does not contain nonzero nilpotent central ideals.

Theorem \ref{rmt} and Lemma \ref{easy} together give the following.

\begin{corollary}\label{ts}
    If the ring  $R$ is equal to  its commutator ideal, then a Jordan homomorphism $\varphi:R\to R'$ is splittable if and only if $\varphi$ is the sum of a homomorphism and an antihomomorphism (on the whole $R$).
\end{corollary}

It is easy to see that the ring $R=M_n(S)$, where $n\ge 2$ and $S$ is a unital ring, is equal to its commutator ideal. 
Corollary \ref{ts} therefore puts  Propositions \ref{rabimjr} and \ref{rabimjr2} in a new light.

The potential usefulness of Corollary \ref{ts} is that, generally speaking, it should be easier to prove that $ \varphi$ is splittable than to prove that it is the sum of a homomorphism and an antihomomorphism.

\begin{remark}
 If 
the commutator ideal $K$ is not equal to $R$ and 
$\varphi$ is the sum of a homomorphism and an antihomomorphism only on  $K$, there is no reason why $\varphi$  should be splittable. 
Just consider Example \ref{gras}: $\varphi$ may not be splittable, but is  a  homomorphism on  $K=Fab$.    
\end{remark}

The next corollary of Theorem \ref{rmt} follows from Proposition \ref{rabim}.

\begin{corollary}
   \label{racac} 
If the ring $R'$ is $2$-torsion free and semiprime, then every 
 surjective Jordan homomorphism $\varphi:R\to R'$ is the sum of a homomorphism and an antihomomorphism on the commutator ideal $K$ of $R$.
\end{corollary}

Proposition \ref{rabimr} shows that the surjectivity assumption can be dropped
if $R'$ is reduced.

\begin{corollary}
   \label{racam} 
If the ring $R'$ is reduced, then every 
 Jordan homomorphism $\varphi:R\to R'$ is the sum of a homomorphism and an antihomomorphism on the commutator ideal $K$ of $R$.
\end{corollary}

Recall from
Example \ref{mexa}  that under the assumptions of  either Corollary \ref{racac}  or Corollary \ref{racam},  $\varphi$ is  not necessarily  the sum   of a homomorphism and an antihomomorphism on the whole $R$. Our next goal is to show that that this is true under the additional assumption that $K$ is a direct summand.
This may seem  quite restrictive, but is justified by Example  \ref{mexa0}. 

First a lemma.

\begin{lemma}\label{lsh} Let $\varphi:R\to R'$ be a splittable Jordan homomorphism. If elements $x$ and $y$  from $R$ commute, then
$\varphi(xy)=\varphi(x)\varphi(y)=\varphi(y)\varphi(x)$.  
\end{lemma}

\begin{proof}
From $xy=yx$ we see that 
  $\{x,y\}=\langle y,x\rangle = - \langle x,y\rangle.$
  As $\{x,y\}\in V_\varphi$, $\langle x,y\rangle\in W_\varphi$, and  $\varphi$ is splittable,   $\{x,y\}=\langle x,y\rangle=0$ follows.
\end{proof}

\begin{corollary}\label{cmalo}
    Let $\varphi:R\to R'$ be a surjective Jordan homomorphism. If  the commutator ideal $K$ of $R$ is a direct summand and $R'$ is a $2$-torsion free semiprime ring, then 
$\varphi$ is the  direct sum of a homomorphism and an antihomomorphism (from $R$ onto $R'$). \end{corollary}

\begin{proof}
    Our assumption is that there exists
    an ideal $C$ of $R$ such that $R=K\oplus C$. For any $c,c'\in C$,
    $[c,c']\in K\cap C=\{0\}$, so $C$ is commutative.  Lemma \ref{lsh}, together with Proposition \ref{rabim}, tells us  
    that the restriction of $\varphi$ to $C$ is a homomorphism. On the other hand, Corollary \ref{racac} tells us that $\varphi$ is the sum of a homomorphism $\varphi_1$ and an antihomomorphism $\varphi_2$ on $K$.
    Define $\Phi_1,\Phi_2:R\to R'$ by
    $$\Phi_1(k + c)=\varphi_1(k)+\varphi(c)\quad\mbox{and}\quad \Phi_2(k + c)=\varphi_2(k)$$
    for all $k\in K$ and $c\in C$. It is clear that $\varphi=\Phi_1+\Phi_2$ and that  $\Phi_2$ is an antihomomorphism.
    Since $kc=ck=0$ for all $k\in K$ and $c\in C$, it follows from conditions
    (c) and (d) of Definition \ref{d1} that 
$$\varphi_1(k)\varphi(c)= \varphi(c)\varphi_1(k) =0\quad\mbox{and}\quad \varphi_2(k)\varphi(c)= \varphi(c)\varphi_2(k) =0.$$ 
The first equation implies that
  $\Phi_1$ is a homomorphism, and from the second equation we infer that
\begin{equation}\label{dvafi}
    \Phi_1(R)\Phi_2(R)=\Phi_2(R)\Phi_1(R)=\{0\}.
\end{equation}  Consequently, 
$$\Phi_1(xy)=\Phi_1(x)\varphi(y)=\varphi(x)\Phi_1(y)\quad\mbox{and}\quad 
\Phi_2(xy)=\Phi_2(y)\varphi(x)=\varphi(y)\Phi_2(x)
$$ for all $x,y\in R$. 
Hence, $    \Phi_1(R)$ and $\Phi_2(R)$ are ideals of $R'$
(because $\varphi$ is surjective). As $R'$ is semiprime, \eqref{dvafi} implies that     $\Phi_1(R)\cap \Phi_2(R)=\{0\}$. All conditions of Definition \ref{d1} are thus fulfilled, so $\varphi$ is the sum of the homomorphism $\Phi_1$ and the antihomomorphism $\Phi_2$ (on $R$).  By Lemma \ref{ll}, this sum is direct.
\end{proof}

Every von Neumann algebra $R$ satisfies the condition of Corollary \ref{cmalo}, i.e., 
its  commutator ideal is a  direct summand (see, e.g., \cite[Lemma 2.6]{B91}). The following thus holds.

\begin{corollary}\label{cma}
If $R$ is a von Neumann algebra and  $R'$ is a  complex semiprime algebra, then every
surjective Jordan homomorphism
    $\varphi:R\to R'$  is the  direct sum of a homomorphism and an antihomomorphism (from $R$ onto $R'$). \end{corollary}

    Since C$^\ast$-algebras 
are semiprime, Corollary \ref{cma} generalizes Kadison's result \cite[Theorem 10]{K} mentioned in the introduction (and also generalizes \cite[Corollary 5.4]{B2}). We remark that $R'$ does need to be an algebra over the complex field, but  involving  general rings  in the context of operator algebras seems slightly unnatural. What is more important is that, unlike in \cite{K}, we did not require that $\varphi$ preserves adjoints.

Our last corollary in this section is an analog of Corollary \ref{cmalo} for reduced rings.

\begin{corollary}   \label{raca2} 
    Let $\varphi:R\to R'$ be a  Jordan homomorphism. If  the commutator ideal $K$ of $R$ is a direct summand and $R'$ is a reduced ring, then 
$\varphi$ is the   sum of a homomorphism and an antihomomorphism (on the whole $R$). \end{corollary}

\begin{proof}
    The proof is almost the same as the proof of Corollary \ref{cmalo}. Only two changes are necessary: at the beginning we refer to Proposition \ref{rabimr} and  Corollary    \ref{racam} instead of Proposition \ref{rabim} and Corollary \ref{racac}, and at the end, when concerned with condition (b) of Definition \ref{d1}, we observe that \eqref{dvafi}  yields $\Phi_1(R)R'\Phi_2(R)=\{0\}$ (because $R'$ is reduced), which further implies that the ideal generated by
    $\Phi_1(R)$ has trivial intersection with the ideal generated by
    $\Phi_2(R)$.
\end{proof}

 We finally remark that the simplest situation in which Corollaries \ref{cmalo} and \ref{raca2} are applicable is when the ring $R$ coincides with its commutator ideal $K$.
    The class of rings with this property  is quite large. In particular, it is closed under direct sums and homomorphic images. Further, if $A$ is an algebra from this class, then  so is $A\otimes B$ for every unital algebra $B$. Another example is a unital ring $R$ having an idempotent $e$ such that $R$ is equal to the ideal generated by $e$ as well as to the ideal generated by $1-e$. Rings with this property appear in the classical study of Jordan homomorphisms \cite{Jac, Mar}.

   \section{The image of the commutator ideal}\label{s5}
   Theorem \ref{rmt} raises the question of what can be said about a
Jordan homomorphism $\varphi:R\to R'$ which  is the sum of a homomorphism and an antihomomorphism on the commutator ideal $K$. 
Does $\varphi$ maps $K$ to the commutator ideal $K'$ of $R'$? 
In  general, the answer is negative. Look again at Example \ref{gras}: the Jordan automorphism $\varphi$ is a homomorphism on the commutator ideal 
$K=Fab$, but does not (necessarily) map $K$ to itself.

In the next theorem, we examine the action of $\varphi$ on some ideals that are contained in 
$K$, in particular on the ideal generated by the Lie ideal
 $[[R,R],R]$.  This Lie ideal  is  intimately connected  with Jordan homomorphisms, which is evident from the equation \eqref{ch} stating that every Jordan homomorphism $\varphi:R\to R'$ satisfies 
$\varphi([[x,y],z])=
     [[\varphi(x),\varphi(y)],\varphi(z)]$ for all $x,y,z\in R$.

\begin{theorem}\label{mt}
Suppose a Jordan homomorphism
$\varphi:R\to R'$ is the sum of  a homomorphism $\varphi_1$ and an antihomomorphism $\varphi_2$ on the commutator ideal $K$
of $R$. Then:
\begin{enumerate}
    \item[1.]  $\varphi_1(K^2)\subseteq K'$ and $\varphi_2(K^2)\subseteq K'$, where $K'$ is the commutator ideal of $R'$.
  \item[2.] $\varphi$ maps the ideal $L$ of $R$ generated by  $[[R,R],R]$ to the ideal  $L'$
  of $R'$ generated by $[[R',R'],R']$.
Moreover, if $\varphi$  is surjective, then $\varphi(L)=L'$.
\end{enumerate}
\end{theorem}

\begin{proof}
    1. 
Take $x,y,z,t,w\in R$, write
$q=[x,y]z[t,w]$, and observe that 
$$q=[[x,y],z][t,w] +
 z[x,y[t,w]] - zy[x,[t,w]].$$
Using condition (c) of Definition \ref{d1}, we thus have
\begin{align*}
\varphi_1(q) =\varphi([[x,y],z])\varphi_1([t,w])+ \varphi(z)\varphi_1([x,y[t,w]]) -\varphi(zy)\varphi_1([x,[t,w]]).
\end{align*}
Each of the three terms lies in $K'$. For the first term this follows from  \eqref{ch}, for the second term from $\varphi_1([x,y[t,w]])= [\varphi(x),\varphi_1(y[t,w])]$ (by condition (c)), and for the third term from 
 $ \varphi_1([x,[t,w]]) = [\varphi(x), \varphi_1 ([t,w])]$ (also by  condition (c)).  Thus, $\varphi_1(q)\in K'$ (observe that this also holds for $q=[x,y][t,w]$). Therefore,
$$\varphi_1(rqs)=\varphi(r)\varphi_1(q)\varphi(s)\in K'$$ for all $r,s\in R$ (and similarly, $\varphi_1(rq), \varphi_1(qs)\in K'$). This proves 
that $\varphi_1(K^2)\subseteq K'$. The proof that $\varphi_2(K^2)\subseteq K'$ is analogous.

2. Note that \eqref{ch} shows that
\begin{equation}\label{alo}
    \varphi([[R,R],R])\subseteq [[R',R'],R'],
\end{equation}
and if $\varphi$ is surjective, then this inclusion becomes the equality.
Recall from Remark \ref{rli} that
every element in $L$ is a sum of elements of the form $w[[x,y],z]$ and $[[x,y],z]$ (a
 similar observation of course holds for $L'$).
Write $$\ell = [[x,y],z]\quad\mbox{and}\quad
 \ell'=\varphi(\ell)=[[\varphi(x),\varphi(y)],\varphi(z)].$$
 Using conditions of Definition \ref{d1}, we see that for every
$w\in R$,
\begin{align*}
  \varphi(w\ell)=&\varphi_1(w\ell) +\varphi_2(w\ell)  = \varphi(w)\varphi_1(\ell) +\varphi_2(\ell)\varphi(w)\\=&\varphi(w)\ell' + [\varphi_2(\ell),\varphi(w)].
\end{align*}
 Since
$$\varphi_2(\ell) = [\varphi(z),\varphi_2([x,y])]\in [R',R'],$$
it follows that 
\begin{equation}
    \label{well}
\varphi(w\ell)- \varphi(w)\ell' \in [[R',R'],R'].\end{equation} Hence
$\varphi(w\ell)\in L'$, which
 proves that $\varphi(L)\subseteq L'$. Also,
 together with \eqref{ch},
 \eqref{well} shows that 
  $ \varphi(w)\ell'\in \varphi(L)$. Therefore,
$\varphi(L)=L'$
  if $\varphi$ is surjective.
 \end{proof}

 In the next and final corollary, 
 we restrict ourselves to  Jordan automorphisms for simplicity. 
Recall first that an ideal $I$ is called an {\em idempotent ideal} if $I^2=I$, and is called a {\em semiprime ideal} if $J^2\subseteq I$ implies $J\subseteq I$ for every ideal $J$.   

\begin{corollary} Let $R$ be a $2$-torsion free semiprime ring such that one of the following conditions is fulfilled:\begin{enumerate}
\item[{\rm (a)}] the commutator ideal $K$ of $R$ is an idempotent ideal,  \item[{\rm (b)}] $K$ is a semiprime ideal, or 
    \item[{\rm (c)}]  $K$  coincides with the ideal $L$ generated by $[[R,R],R]$. \end{enumerate} Then every Jordan automorphism of $R$   is the direct sum of a homomorphism and an antihomomorphism from $K$ onto $K$. 
\end{corollary}

\begin{proof} 
By Corollary \ref{racac}, $\varphi$ is the sum
of a homomorphism $\varphi_1$ and an antihomomorphism $\varphi_2$ on $K$.

 Since $\varphi_i(K^2)=\varphi_i(K)^2\subseteq K$, each of the conditions (a) and (b) implies that $\varphi_i(K)\subseteq K$ for  $i=1,2$. As $\left.\varphi\right|_{K}=\varphi_1+
\varphi_2$, this gives $\varphi(K)\subseteq K$. Now, $\varphi^{-1}$ is also a Jordan automorphism of $R$, so we analogously have     $\varphi^{-1}(K)\subseteq K$. Therefore, $\varphi(K)=K$,
and the desired conclusion  follows from Lemma \ref{ll}.

Assume that (c) holds.
Then
$\varphi(K)=K$  by the second statement of Theorem \ref{mt}. 
Since $$\varphi_1(w[[x,y],z]) = \varphi(w)[\varphi_1([x,y]),\varphi(z)] \in K,$$
we have $\varphi_1(K)=\varphi_1(L)\subseteq K$.
Invoking Lemma \ref{ll} concludes the proof.  
\end{proof}

We leave as an open question whether the
assumption that one of the conditions (a)-(c) holds is superfluous.

\bigskip

\noindent
{\bf Acknowledgment.} The author is grateful to Efim Zelmanov for the discussion that inspired this work.

\end{document}